\documentclass[11pt, DIV=12, a4paper]{scrartcl}

\usepackage{amsthm, amsfonts, amsmath, latexsym, mathrsfs, amssymb}
\usepackage{mathtools}

\usepackage[utf8]{inputenc}
\usepackage{xcolor}
\usepackage{graphicx}
\usepackage[T1]{fontenc}
\usepackage{newtx}

\usepackage{enumitem}

\usepackage{hyperref}
\hypersetup{
	colorlinks=true,
	linkcolor=blue!60!black,
	citecolor=blue!60!black,
	urlcolor=blue!60!black
}

\newtheorem{theorem}{Theorem}[section]
\newtheorem{lemma}[theorem]{Lemma}
\newtheorem{result}[theorem]{Result}

\newtheorem{corollary}[theorem]{Corollary}

\theoremstyle{definition}
\newtheorem{definition}[theorem]{Definition}

\newtheorem{prob}[theorem]{Problem}

\theoremstyle{remark}
\newtheorem{remark}[theorem]{Remark}

\numberwithin{equation}{section}

\newcommand{\bdy}{\partial}
\newcommand{\OM}{\Omega}

\newcommand{\Sone}{\mathbb{S}^1}

\newcommand{\smoo}{\mathcal{C}}

\newcommand{\Cn}{\mathbb{C}^n}

\newcommand{\C}{\mathbb{C}}
\newcommand{\R}{\mathbb{R}}
\newcommand{\N}{\mathbb{N}}

\newcommand{\D}{\mathbb{D}}

\newcommand{\B}{\mathbb{B}}

\newcommand{\F}{\textsf{Fix}}
\newcommand{\hol}{\mathcal{O}}
\newcommand{\A}{\textsf{Aut}}

\usepackage{enumitem}

\setlist[itemize]{noitemsep}
\setlist[enumerate]{noitemsep}

\usepackage[automark]{scrlayer-scrpage}
\clearpairofpagestyles 
\ohead{\headmark} 
\ofoot{\pagemark} 
\makeatletter
\def\blfootnote{\gdef\@thefnmark{}\@footnotetext}
\makeatother

\begin{document}
	\title{Finiteness of the fixed point sets of automorphisms}

	\author{ Bharathi Thiruvengadam \thanks{Bharathi is supported by the University Grants Commission (UGC), India, through the UGC-NET Senior Research Fellowship (SRF).\\
  \href{mailto:212114004@smail.iitpkd.ac.in}{212114004@smail.iitpkd.ac.in}, \href{mailto:bharathit.math@gmail.com}{bharathit.math@gmail.com}\\
Department of Mathematics, Indian Institute of Technology Palakkad,
Palakkad, Kerala- 678623, India} \and Jaikrishnan Janardhanan \thanks{Jaikrishnan is supported by a grant from ANRF: \textbf{ANRF/ARGM/2025/000619/MTR} \\
\href{mailto:jaikrishnan@iitpkd.ac.in}{jaikrishnan@iitpkd.ac.in}\\
Department of Mathematics, Indian Institute of Technology Palakkad,
Palakkad, Kerala- 678623, India} }

	\date{\today}

	\maketitle


	\begin{abstract}
		We investigate the size of fixed point sets of automorphisms of bounded
		domains in $\C^n$. In one complex variable, a nontrivial automorphism
		has at most two fixed points, but in higher dimensions fixed point sets
		need not be discrete. We show, under natural extension hypotheses, that
		discreteness forces finiteness. We also obtain a uniform bound for the
		number of fixed points of automorphisms in compact subgroups whose
		elements admit such extensions. 
	\end{abstract}

	\blfootnote{\textup{2020} \textit{Mathematics Subject Classification}: Primary 32H50;}


	\section{Introduction}

	The study of fixed points of holomorphic self-maps of complex manifolds is of fundamental
	importance in several complex variables, complex geometry and complex dynamics.
	For a holomorphic self-map $f$ of a complex manifold $M$, we define its fixed
	point set by
	\[
		\F(f) := \{ z \in M : f(z) = z \}.
	\]
    
	In this paper we study the following basic question for bounded domains in
	$\C^n$: if an automorphism has a \emph{discrete} fixed point set, must that
	set be finite? More ambitiously, can one bound its cardinality in terms of
	the domain alone? 
    
    In one complex variable the picture is completely understood. On a hyperbolic planar domain, the
	existence of two fixed points is sufficient to ensure that a holomorphic self-map
	is an automorphism. Furthermore, three fixed points forces the map to be the
	identity. This is a classical result for which several proofs are available;
	see
	\cite{maskit1968conformal,Leschinger1978fixed,peschl1979conformal,minda1979fixed,suita1981fixed,fischer1987fixed,krantz2005conformal}.
	Recently, in the paper \cite{Bharathi2025JK}, we gave a new proof of this result
	using the topology of the balls under the Kobayashi distance. Since the only non-hyperbolic planar domains are $\C$ and $\C\setminus\{0\}$, whose automorphisms are M\"obius transformations, one immediately obtains the following

	\begin{result}
		Let $D \subset \C$ be a domain and let $f:D \to D$ be an automorphism that is
		not the identity. Then $\#\F(f) \leq 2$.
	\end{result}
	\medskip

	The higher-dimensional picture is markedly different as the fixed point set of an automorphism is not necessarily
	discrete: for instance, the map $(z,w) \mapsto (-z,w)$ on the bidisk fixes the complex line $\{0\}\times \D$.
	However, the fixed point set of a holomorphic 
    self-map is always an analytic subset of the domain. Indeed, if $D \subset \C^n$ is a domain
    and $f = (f_1,\dots,f_n):D \to D$ is holomorphic then 
    \[
        \F(f) = \{z \in D: f_1(z) - z_1 = f_2(z) - z_2 = \dots = f_n(z) - z_n = 0\}.
    \]
    As we have $n$ equations in $n$ variables, we should expect the fixed point set to be
 generically $0$-dimensional, i.e., a discrete set. In fact, by the inverse function theorem, if a point $p \in \F(f)$ is
    \textbf{not} an isolated point of $\F(f)$ then $1$ must necessarily be an eigenvalue of $df(p)$. 
    Thus, in higher dimensions, the natural question is not whether fixed points are few in an absolute
	sense, but rather whether one can control the size of fixed point sets of
	automorphisms that are \emph{a priori} assumed to be discrete. The following
	remarkable result of Vigué settles the question for domains that are bounded
	and convex. Indeed, the result shows that if the fixed point set of an automorphism
    is discrete then it is a connected discrete set and is therefore a singleton set whenever non-empty.

	\begin{result}
		[Vigué \cite{JPVigue1985}]\label{R:viguecvx} Let $f$ be a holomorphic self-map
		of a bounded convex domain $D \subset \mathbb{C}^{n}$. If $\F(f)$ is nonempty,
		then $\F(f)$ is a holomorphic retract of $D$. In particular, $\F(f)$ is a
		connected complex submanifold of $D$.
	\end{result}

	Beyond the convex case, much less is known. A fundamental result of Fridman--Ma--Vigu\'e
	\cite{fridman2006fixed} shows that for automorphisms of a bounded strictly pseudoconvex
	domain with real-analytic boundary, a discrete fixed point set is necessarily
	finite, and in fact uniformly bounded in cardinality.

	\begin{result}
		[\cite{fridman2006fixed}] \label{R:fridman} Let $D \subset \mathbb{C}^{n}$ ($n
		\ge 1$) be a bounded strictly pseudoconvex domain with real-analytic boundary.
		Then, for every automorphism $f$ of $D$ whose fixed-point set is discrete,
		the set of fixed points of $f$ is finite. Moreover, there exists a constant
		$m = m(D) \in \mathbb{N}$, depending only on $D$, such that
		$\# \mathsf{Fix}(f) \le m$, for all such automorphisms $f$.
	\end{result}
	\noindent
	The proof combines three deep ingredients.
	\begin{enumerate}
		\item The Wong--Rosay theorem \cite{Wong1977, Rosay1979} that characterizes
			the unit ball as the only smoothly bounded strongly pseudoconvex domain
			with non-compact automorphism group.

		\item The holomorphic extension of automorphisms past the boundary for strongly
			pseudoconvex domains with real-analytic boundary (see Result~\ref{R:extension}).

		\item Whitehead's theorem \cite{Whitehead1932} on the existence of geodesically convex
			neighborhoods applied to the Bergman metric.
	\end{enumerate}
	While the proof of Result~\ref{R:fridman} is short and elegant, the method breaks
	down when one does not have holomorphic extension of automorphisms beyond
	the boundary. Furthermore, as the method crucially relies on the fact that
	automorphisms are isometries of the Bergman metric, there is little hope
	that fixed point sets of more general holomorphic mappings can be studied with
	the same method. In the same paper \cite{fridman2006fixed}, the authors also constructed bounded domains
	admitting automorphisms with arbitrarily large finite fixed point sets,
	thereby showing that any general bound must depend on the domain. This leads naturally
	to the following basic problem (similar questions were also raised as unsolved
	problems in \cite{fridman2006fixed}).

	\begin{prob}
		\label{P:main} Let $D \subset \C^{n}$ be a bounded domain.
		\begin{enumerate}
			\item If $f \in \A(D)$ is such that $\F(f)$ is discrete then is it true that
				$\F(f)$ is finite?

			\item Furthermore, does there exist a uniform bound $m := m(D)$ such that
				if $f \in \A(D)$ has a discrete fixed point set, then $\# \F(f) \leq m$?
		\end{enumerate}
	\end{prob}

	The main purpose of this article is to show that both questions admit
	affirmative answers under natural hypotheses satisfied by large classes of
	bounded domains. More importantly, our method of proof is quite different
	from that of \cite{fridman2006fixed} and substantially more flexible. The two crucial observations that led us to
	our method are as follows.
	\begin{enumerate}
		\item The proof of Result~\ref{R:fridman} does \textbf{not} use any
			information about the structure of the fixed point set other than the
			hypothesis that it is discrete. This gave us a new unexplored line of attack.

		\item Whitehead's theorem works equally well in an arbitrary bounded domain to
			show that the fixed points of automorphisms cannot get ``too close'' to each
			other when the fixed point set is discrete. More precisely, given a point $a$
			in the domain $D$, there can be at most $1$ fixed point of an automorphism
			with a discrete fixed point set in a geodesically convex neighborhood around $a$.
			Thus, it suffices to understand what happens when fixed points accumulate at 
            a point of the boundary.
	\end{enumerate}

	In contrast to the proof of Result~\ref{R:fridman}, our method relies heavily
	on the fact that the fixed point set of a holomorphic self-map is a complex
	analytic set and that it is a complex submanifold when the domain is bounded. In
	order to understand what happens when fixed points accumulate at the boundary,
	we rely on Results~\ref{R:Cartaneigen} and \ref{R:vigue} below that show that a point
    $z_0$ in the fixed point set of an automorphism $f$ is an isolated point if and only if $1$ is not an eigenvalue of $df(z_0)$. Furthermore, the eigenvalues of $df(z_0)$ must lie on the unit 
    circle. Therefore, we can perturb the eigenvalues of the Jacobian matrix by iterating, 
    and they would continue to lie on the unit circle. It is now natural to wonder how much control we can have on
	the eigenvalues under iteration. This led us to look at standard theorems in
	Diophantine approximation like Kronecker's theorem and Weyl's
	equidistribution theorem. The culmination of our investigation is the \emph{arc-avoidance}
	lemma (Lemma~\ref{T:arc}). More precisely, the lemma shows that there is an arc about $1$ on the
	unit circle (the length of this arc depends only on the ambient dimension) 
    and an integer $k \ge 1$ such that the eigenvalues of $df^k(z_{0})$ avoid
	this arc. This arc-avoidance lemma
	is closely related in spirit to questions surrounding the famous lonely runner
	conjecture. By appropriately perturbing the eigenvalues by iterating and using the arc-avoidance lemma, we are able to
	force a contradiction. We believe that this method of perturbing the eigenvalues
	by iteration is of independent interest and may be useful in other problems concerning
	fixed points of holomorphic maps. For instance, the second-named author has used
	a similar idea in an earlier work that involved proper holomorphic mappings \cite{jaikrishnan2023} where eigenvalues
    were perturbed by composing with the symmetries of the domain.
	It is also worthwhile noting that though Whitehead's theorem was instrumental
	in guiding us to our method, we do not use it in the proof of our main result.
	Instead, we use a very simple version of Rouch\'e's theorem in higher dimensions
	to show that fixed points cannot be ``too close'' to each other.

	We prove two main results in this article. They are intentionally formulated
	as very general theorems in order to illustrate the scope of our methods. Our first
	result does \textbf{not} require any explicit regularity assumptions on the domain, but
    instead we require the automorphisms to extend beyond the boundary. 

	\begin{theorem}
		\label{T:main} Let $D\subset \C^{n}$ be a bounded domain.
		\begin{enumerate}
			\item If $f \in \A(D)$ is an automorphism that extends holomorphically beyond
				the boundary and whose fixed point set in $D$ is discrete then $\F(f)$
				is finite.

			\item Let $H$ be a compact subgroup of $\A(D)$. Assume that every
				$f \in H$ extends holomorphically beyond $\overline{D}$. Then we can find
				a number $m\in \mathbb{N}$ (depending only on $D$ and $H$) such that if $f
				\in H$ has a discrete fixed point set, then the fixed point set of $f$
				is a finite set of cardinality at most $m$.
		\end{enumerate}
	\end{theorem}

	Our second theorem is for bounded domains with smooth boundary.
	Here it is unnatural to assume that automorphisms extend holomorphically
	beyond the boundary. Instead, we assume that they extend as
	diffeomorphism to the closure and that we have ``uniformity'' in the extension.  Notably, no pseudoconvexity assumption
	is built into the statement of our theorem.

	\begin{theorem}
    		\label{T:meta} Let $D \subset \mathbb{C}^{n}$ be a bounded domain with
		smooth boundary that satisfies the following conditions: 
    \begin{enumerate}
      \item Each $f \in \A(D)$ extends as a diffeomorphism to $\overline{D}$.
      \item If $f_n \to f$ in $\A(D)$, then $f_n \to f$ in the $\smoo^{\infty}$-topology on $\overline{D}$ (by this we mean each component of $f_n$ converges to the corresponding component of $f$ in $\smoo^{\infty}$-topology on $\overline{D}$).
    \end{enumerate}
    Then for each compact subgroup $H \subset \A(D)$ we can find
				a number $m\in \mathbb{N}$ (depending only on $D$ and $H$) such that if
				$f	\in H$ has a discrete fixed point set, then the fixed point set of
				$f$	is a finite set of cardinality at most $m$.

	\end{theorem}

	The above two theorems have broad applicability and yield a complete solution to
	Problem~\ref{P:main} for large classes of domains. In particular, both theorems
	subsume Result~\ref{R:fridman}. It is also worthwhile highlighting that our
	methods work for more general classes of mappings as well. We now state some concrete instantiations
	of our main theorems. For more details and examples, see
	Section~\ref{S:examples}. 

	\begin{corollary}
		\label{C:main} Let $D \subset \C^{n}$ be a bounded balanced domain. If
		$f \in \A(D)$ has a discrete fixed point set then $\F(f)$ is finite.
	\end{corollary}
    \begin{corollary}\label{C:main2}
        Let $D \subset \C^n$ be a bounded domain with smooth, real-analytic boundary
        and let $f \in \A(D)$ have a discrete fixed point set. Suppose that either $n = 2$ or that $f$
        extends continuously to $\overline{D}$ or that $D$ is pseudoconvex. Then $\F(f)$ is finite. Furthermore, if $n = 2$ then there is a positive integer $m$, depending only on $D$, such that for
		every $f \in \A(D)$ with discrete fixed point set, the number of fixed points
		of $f$ is at most $m$.
    \end{corollary}

	\begin{corollary}
		\label{C:meta} Let $D \subset \C^{n}$ be a bounded strongly pseudoconvex domain.
		Then there is a positive integer $m$, depending only on $D$, such that for
		every $f \in \A(D)$ with discrete fixed point set, the number of fixed points
		of $f$ is at most $m$.
	\end{corollary}

	\begin{corollary}
		\label{C:meta2} Let $D \subset \C^{n}$ be a bounded domain with smooth
		boundary that satisfies condition $R$. If $f \in \A(D)$ has a discrete fixed
		point set then $\F(f)$ is finite.
	\end{corollary}

	The paper is organized as follows. In Section~\ref{S:prereqs} we collect the
	preliminary material used throughout the paper.  We prove
	Theorems~\ref{T:main} and \ref{T:meta} in Section~\ref{S:proof}.
    Several classes of domains to which these theorems apply, including the strongly pseudoconvex
	case and other smoothly bounded domains satisfying suitable extension
	hypotheses are discussed in Section~\ref{S:examples}. In this section, we also discuss 
    results for more general maps and speculate on other possible extensions of our main results.







	\section{Necessary results}
	\label{S:prereqs}

	In this section, we discuss some results and lemmas that we will need in the proof
	of our main results. For a bounded domain
	$D$ in $\C^{n}$, we say that an automorphism $f$ of $D$ extends \textit{holomorphically
	beyond the boundary} (or \textit{past the boundary} or \textit{beyond $\overline{D}$}) if there exists an open neighborhood $U$ of $\overline{D}$
	and a holomorphic map $\tilde{f}: U \to \C^{n}$ such that $\tilde{f}|_{D} = f$.
	If, furthermore, $\tilde{f}$ is injective on $U$ we say $f$ \textit{extends
	biholomorphically beyond the boundary}  or that $f$ \textit{extends
	biholomorphically to $U$}. The first ingredient we need in the proof of Theorem~\ref{T:main}
	is the uniform extension of automorphisms beyond the boundary. This is Theorem~\ref{thm: uniform extension} below.
	We will use Vitali's theorem in the proof which we recall now.
    \begin{definition}
       A subset $S$ of a domain $D \subset \mathbb{C}^{n}$ is called a \textit{set of
	   uniqueness} if every holomorphic function $f$ on $D$ that vanishes on $S$ is identically
	   zero on $D$. 
    \end{definition}

    For example, any nonempty open subset of $D$ is a set of uniqueness
	   by the identity Theorem.

	\begin{result}[Vitali's theorem]
  Let $\{f_n\}$ be a sequence of holomorphic functions on a domain $D$, and let
  $S \subset D$ be a set of uniqueness. Suppose that $\{f_n\}$ is uniformly bounded on $D$
  and that $f_n(a)$ converges for every $a \in S$. Then $\{f_n\}$ converges uniformly on
  every compact subset of $D$.
\end{result}

\noindent
See {\cite[Chapter~1, Proposition~7]{Narasimhan1995book}}.

	We now state some results about the structure of the automorphism group of a
	bounded domain due to H. Cartan. 
\begin{result}[Cartan]\label{R:cartan}
Let $D$ be a bounded domain in $\C^n$. Let $\{f_k\}$ be a sequence of automorphisms
of $D$ converging uniformly on compact subsets to a map $f:D\to\C^n$. Then the following
are equivalent:
\begin{enumerate}
	\item $f \in \A(D)$.
	\item $f(D) \not\subset \partial D$.
	\item There exists a point $a \in D$ such that the determinant of the Jacobian matrix
	of $f$ at $a$ is nonzero.
\end{enumerate}
\end{result}

\noindent
See {\cite[Chapter~5, Theorem~4]{Narasimhan1995book}}.

\begin{result}[Cartan]
Let $D \subset \C^n$ be a bounded domain. Then the group $\A(D)$, endowed with the
compact-open topology, carries the structure of a Lie group such that the natural map
$\A(D)\times D \to D$ is real-analytic.
\end{result}

\noindent
See {\cite[Chapter~9]{Narasimhan1995book}}.

	The following result concerns the holomorphic extension of smooth $CR$-diffeomorphisms
	beyond the boundary.

	\begin{result}
		[\cite{Baouendietal1985}] \label{R:extension} Let $D$ be a smoothly bounded domain
		in $\C^{n}$ $(n\geq 2)$ with real analytic boundary. Suppose that either $D$
		is pseudoconvex or condition $R$ holds on $D$, then any $f\in \textsf{Aut}(D)$
		extends biholomorphically beyond the boundary
	\end{result}

     Using this boundary extension result, Chen and Jang
	proved a uniform extension theorem for automorphisms of $D$ when
	$\A(D)$ is compact.

	\begin{result}[\cite{chen1999Jang}]\label{R:Uniform extension} Let $D$ be a domain
		satisfying the hypotheses of Result~\ref{R:extension}. Assume
		that the automorphism group $\textsf{Aut}(D)$ is compact. Then there exists
		an open neighborhood $U$ of $\overline{D}$ such that each
		$f \in \textsf{Aut}(D)$ extends as biholomorphically to $U$.
	\end{result}

	A careful reading of the proof of the above result reveals that the same conclusion
	holds under only the assumption that automorphisms extend holomorphically beyond
	the boundary with no assumptions whatsoever on the domain other than boundedness.
	To the best of our knowledge, this simple, but quite general, and very useful observation has not
	been recorded in the literature. This observation combined with simple
	topology shows that given a compact subgroup of the automorphism group, we can
	find a single open set that contains $\overline{D}$ to which all automorphisms
	in the subgroup extend as automorphisms. This, in the restricted setting of strongly
	pseudoconvex domains with real-analytic boundary, was the first ingredient in
	the proof of Result~\ref{R:fridman} and is the first ingredient in the proof
	of Theorem~\ref{T:main}.

	\begin{theorem}
		\label{thm: uniform extension} Let $D$ be a bounded domain in $\C^{n}$ and let
		$H$ be a compact subgroup of $\A(D)$. Assume that every $f\in H$ extends
		holomorphically beyond $\overline{D}$. Then there is a single domain $U \supset \overline{D}$ such that all automorphisms in $H$ extend to be automorphisms
		of $U$.
	\end{theorem}

	\begin{proof}
		First we show that each $f \in H$ actually extends biholomorphically beyond the
		boundary. Let $F : U \to \mathbb{C}^{n}$ be a holomorphic extension of $f$
		to an open neighborhood $U$ of $\overline{D}$, and let $G : V \to \mathbb{C}^{n}$
		be a holomorphic extension of $f^{-1}$ to an open neighborhood $V$ of $\overline
		{D}$. Since $f$ is a biholomorphism from $D$ onto $D$, we have
		\[
			F(\overline{D}) = \overline{D}\subset V .
		\]
		Since $F$ is continuous, the set $F^{-1}(V)$ is open in $U$. Hence
		\[
			U' = U \cap F^{-1}(V)
		\]
		is an open neighborhood of $\overline{D}$. For every $z \in U'$, we have $F(z
		) \in V$, and therefore the composition $G \circ F$ is well-defined and
		holomorphic on $U'$. Let $W$ be the connected component of $U'$ that contains
		$\overline{D}$. Then $W$ is an open and connected neighborhood of $\overline{D}$,
		and the map $G \circ F : W \to \mathbb{C}^{n}$ is holomorphic. For $z \in D$,
		we have
		\[
			G(F(z)) = G(f(z)) = z .
		\]
		Thus $G \circ F$ agrees with the identity map on $D$. By the identity
		theorem for holomorphic maps in $\mathbb{C}^{n}$, it follows that
		\[
			G \circ F|_{W} = \mathrm{id}_{W} .
		\]
		Hence $F|_{W}$ is injective and $F(W) \subset V$. Therefore
		\[
			F|_{W} : W \to F(W)
		\]
		is a biholomorphism, and $F(W)$ is an open neighborhood of $\overline{D}$.

		Define for each $n \in \N$
		\[
			D_{j} := \{z \in \C^{n}: \textsf{dist}(z, \overline{D}) < 1/j\}.
		\]
		For each $f \in H$, we can find an open set $D_{j}$ such that $f$ extends to
		be an injective and bounded holomorphic map in $D_{j}$. Set
		\[
			E_{j,m}:= \{g \in H: g \text{ extends biholomorphically to $D_j$ and is bounded
			by $m$ on $D_j$} \}.
		\]
		Then $H = \bigcup_{j,m}E_{j,m}$. Notice that each $E_{j,m}$ is closed. For if
		$f_{k} \in E_{j,m}$ and $f_{k} \to f$ uniformly on compact subsets of $D$,
		then as the $f_{k}$'s are uniformly bounded, Vitali's theorem implies that $f
		_{k} \to f$ uniformly on compact subsets of $D_{j}$. Hurwitz's theorem (see
		\cite[p.80]{Narasimhan1995book}) now
		guarantees that $f$ is injective on $D_{j}$ and it is clear that $f$ is also
		bounded by $m$. Thus $f \in E_{j,m}$.

		As $\textsf{Aut}(D)$ is a Lie group, it is metrizable. Thus $H$ can be viewed
		as a complete metric space. By the Baire category theorem, there is an open set
		$V$ in $H$ that is fully contained in some $E_{j,m}$. We have found an open
		subset $V$ of $H$ such that every automorphism in $V$ extends biholomorphically 
        to a fixed $D_{j}$ and such that the extended map
		is bounded by a fixed $m$. We can now translate this open set to any given automorphism
		in $H$. Any automorphism in the translated set will also extend as a bounded
		injective holomorphic map to some (possibly different) $D_{j}$. The constant
		$m$ can change as well. Nevertheless, by the compactness of $H$, we can cover
		$H$ by finitely many such open sets, and thus we can find a single domain
		$\OM \Supset \overline{D}$ to which each automorphism in $H$ extends as a
		bounded injective holomorphic map. In fact, we have obtained a uniform bound.
		It is now easy to see that the topology on $H$ inherited from $\A(D)$ coincides
		with the one inherited from $\hol(\OM,\C^{n})$.

		Now we will show each $h \in H$ extends to be an automorphism of a fixed $U$
		that contains $\overline{D}$. Fix $f \in H$ and let
		$U_{2} \Subset U_{1} \Subset \OM$ be domains that contain $\overline{D}$. Choose
		an open set $V$ around $f$ in $H$ such that for $g \in V$, we have
		$g(\bdy U_{1}) \cap f(\overline{U}_{2}) = \emptyset$. We claim $f(U_{2}) \subset
		g(U_{1})$ for all $g \in V$. If not, then there is a point $z \in f(U_{2})$ such that $z \notin
		g(U_{1})$. We can find a path $\gamma$ in $f(U_{2})$ that connects $z$ to a
		point in $D$. But $D \subset g(U_{1})$. This means that $\gamma$ must
		intersect $g(\bdy U_{1})$. This is a contradiction. We have now shown that
		$g(U_{1})$ contains a fixed open set, namely $f(U_{2})$, that in turn
		contains $\overline{D}$.

		As $H$ is compact, we can find a single open set $Q$ that contains $\overline
		{D}$ and such that each $h \in H$ satisfies $Q \subset h(U_{1})$. Let $U$ be
		the connected component that contains $D$ of the interior of the
		intersection of all the sets $h(U_{1})$ as $h$ ranges over $H$. Then $U$ is a
		domain that contains $Q$. Furthermore, for each $h \in H$, we have
		$h(U) \subset U$. In particular, $h^{-1}(U) \subset U$. Thus $h$ is an automorphism
		of $U$.
	\end{proof}

	We need one more result that bears H. Cartan's name.

	\begin{result}
		[Cartan--Carathéodory --Kaup--Wu Theorem, \cite{Wu1967}, Theorem 2.1.21 of \cite{Abate1989notes}]\label{R:Cartaneigen}
		Let $D$ be a bounded domain in $\Cn$ and $f\in \hol(D,D)$ with $f(z_{0})=z_{0}$,
		then
		\begin{enumerate}
			\item the set of all eigenvalues of $df(z_0)$ are contained in the
				closed unit disc $\overline{\D}$;

			\item $|\det df(z_0)|\leq 1$;

			\item $df(z_0)=I$, the identity linear map if and only if $f$ is the
				identity map;

			\item $T_{z_0}D$ admits a $df(z_0)$-invariant splitting $T_{z_0}D=L_{U}\oplus
				L_{N}$ such that the spectrum of $df(z_0)|_{L_N}$ is contained in $\D$,
				the spectrum of $df(z_0)|_{L_U}$ is contained in $\bdy \D$ and
				$df(z_0)|_{L_U}$ is diagonalizable;

			\item $|\det df(z_0)| = 1$ if and only if $f$ is an automorphism.
		\end{enumerate}
	\end{result}

	In our proof of the main theorems, we will use the fact that the fixed point set
	$\F(f)$ is a complex submanifold of $D$ whenever $D$ is a bounded domain. This
	was proved by Vigué using ideas of Bedford and H. Cartan.

	\begin{result}
		[Vigué \cite{Vigue1986}]\label{R:vigue} Let $D$ be a bounded domain in
		$\C^{n}$ and let $f: D \rightarrow D$ be a holomorphic mapping. Then $\F(f)$
		is a complex submanifold of $D$. If $a \in \F(f)$, its tangent space $T_{a}(\F(f))$ is given by
		\[
			\left\{v \in \mathbb{C}^{n} \mid df(a) \cdot v=v\right\} .
		\]
	\end{result}

	We now state a version of Rouch\'e's theorem in several complex variables
	that we will use in the proof of our main results.

	\begin{result}[Rouché's Theorem in higher dimensions]
		\label{R:Rouché} Let $f, g :U \to \C^{n}$ be two holomorphic maps defined on a
		neighborhood $U$ of some ball $\overline{\B}(a,r) \subset \C^n$ such that
		\[
			\|g(z)-f(z)\| < \|f(z)\|, \ \text{for all }z \in \bdy \B(a,r).
		\]
		Then $f$ and $g$ have the same number of zeros in $\B(a,r)$, counted with
		multiplicities.
	\end{result}

	\begin{remark}
		A more general version for arbitrary domains instead of
		balls can be found in Lloyd's paper \cite{Lloyd1975} and the proof relies on the notion of the topological degree of a
		map. The book \cite[Appendix~B]{Milnor1968} deals only with domains with smooth boundary and the version stated above follows quite easily from the first two lemmas in Appendix
		B there. 
	\end{remark}

    \begin{remark}
        We will \textit{not} give the definition of multiplicity here but in the sequel we require only one fact: 
        the multiplicity at a point where the derivative is non-singular is $1$.
    \end{remark}

	We now come to the arc-avoidance lemma. This lemma is related to
	the famous lonely runner conjecture in Diophantine
	approximation. While this conjecture is still not fully settled, the integer analogue
	of the conjecture, dubbed the lonely rabbit problem has been fully solved. For
	$n \in \mathbb{N}$, the \emph{Rabbit value} of $n$ is defined by
	\[
		\textsf{Rab}(n) := \inf_{v_1,\dots,v_n \in \mathbb{R}\setminus\mathbb{Z}}\; \sup
		_{m \in \mathbb{Z}}\; \min_{1 \le j \le n}\| m v_{j} \|,
	\]
	where $\|x\| := \min_{m\in\mathbb{Z}}|x-m|$ denotes the distance from $x$ to the
	nearest integer. In \cite{Wills1968},
	Wills proved that
	\[
		\textsf{Rab}(1) = \frac{1}{3}, \qquad \frac{1}{2n^{2}}\le \textsf{Rab}(n) \le
		\frac{1}{w(n)}\quad \text{for all }n \ge 2.
	\]
	For our purposes, all we need is the fact that $\textsf{Rab}(n) > 0$. The
	following lemma is easily seen to be equivalent to the claim that
	$\textsf{Rab}(n) > 0$.

	\begin{lemma}\label{T:arc} Fix an integer $n\ge 1$. Then there exists an open arc
		$L\subset \Sone$ containing $1$, depending only on $n$, such that for every choice
		\[
			\lambda_{1},\dots,\lambda_{n}\in \Sone\setminus\{1\}
		\]
		there exists $k\in \N$ with
		\[
			\lambda_{j}^{k}\notin L\qquad (j=1,\dots,n).
		\]
	\end{lemma}

  The final tool we need is the continuity of eigenvalues. Since this is
  standard, we will not give a proof here and use the following lemma without
  explicit mention in the proof of our main results.

\begin{lemma}[Continuity of the spectrum in the Hausdorff metric]
Let $\mathcal K(\C)$ denote the set of all nonempty compact subsets of $\C$,
equipped with the Hausdorff metric $d_H$. Then the map
\[
\sigma : M_n(\C)\to \mathcal K(\C), \qquad A\mapsto \sigma(A),
\]
is continuous. Equivalently, if $A_k\to A$ in $M_n(\C)$, then
\[
d_H\bigl(\sigma(A_k),\sigma(A)\bigr)\to 0.
\]
\end{lemma}

	\section{Proofs of main results}
	\label{S:proof}

	\noindent
	\textbf{Proof of Theorem \ref{T:main}}

	The first part follows from standard facts about analytic varieties. Let $f$ be an automorphism of $D$ with a
	discrete fixed point. Let $g$ be the holomorphic extension of $f$ to $U$ where $U$
	is some bounded domain that contains $\overline{D}$. Suppose that $\F(f)$ is
	infinite in $D$. Then there exists a sequence $\{z_{n}\} \subset \F(f)$ converging
	to a point $z_{0} \in \partial D$. Note that both $z_{n}$ and $z_{0}$ are
	fixed points of $g$. Furthermore, the point $z_{0}$ is \textbf{not} an isolated
	point $\F(g)$. In a small open set $V$ around $z_0$, $\F(g)$ decomposes into a finite union of irreducible analytic varieties. Let us call the component that contains infinitely many points of the \textit{set} 
    $\{z_n\}$ as $X$. It follows from standard results on the structure of
	analytic varieties that $\textsf{dim}_{z_0}(X) \geq 1$ (see \cite[p.199]{Loja1991}).
	This contradicts the fact that $\F(f)$ is discrete. It follows that $\F(f)$ is
	finite.
	\medskip

	For the uniform bound, first choose $U \Supset \overline{D}$ so that all automorphisms
	in $H$ extend to be automorphisms of $U$. It suffices to show that for each
	point $z_0 \in U$ there exists a small ball centered at $z_0$ such that each
	$f \in H$ has at most one fixed point in the small ball. This is actually immediate
	from Whitehead's theorem, but we will give a new proof that does
	not rely on this fact. Assume, to the contrary, that there exists a point $z_{0}
	\in U$ for which there is a sequence $\{f_{k}\} \subset H$ such that each $f_{k}$
	has at least two distinct fixed points $a_{k}$ and $b_{k}$ in the Euclidean ball
	$\B(z_{0},1/k)$. Since each $f_{k}$ is an automorphism of $D$, the Cartan--Carathéodory--Kaup--Wu
	theorem (Result \ref{R:Cartaneigen}) implies that
	$df_{k}(a_{k})$ and $df_{k}(b_{k})$ have all eigenvalues on the unit circle  and, by Result~\ref{R:vigue},
    none of them equals $1$. Denote
	the eigenvalues of $df_{k}(a_{k})$ by
	\[
		\lambda_{1k},\lambda_{2k},\ldots,\lambda_{nk}.
	\]
	All these eigenvalues lie on the unit circle, By Lemma~\ref{T:arc}, there exists
	a small arc $L$ on the unit circle containing $1$, depending only on $n$ and
	positive integers $m_{k}$ such that
	\[
		\lambda_{jk}^{m_k}\notin L, \qquad j=1,2,\ldots,n.
	\]
	Consider now the sequence of automorphisms $f_{k}^{m_k}$. Observe that $f_{k}^{m_k}$
	has at least two distinct fixed points $a_{k}$ and $b_{k}$ in the Euclidean ball
	$B(z_{0},1/k)$, and the eigenvalues of $df_{k}^{m_k}(a_k)$ are $\lambda_{jk}^{m_k}$.

	Since $U$ is bounded, we can apply Montel's theorem to pass to a subsequence and assume that the sequence $f_{k}^{m_k}$ converges to a holomorphic map $f \in \hol(U,U)$. Since $a_{k}$
	and $b_{k}$ are fixed points of $f_{k}^{m_k}$, we have $f(z_{0})=z_{0}$.
	Suppose that $z_{0}$ is an isolated fixed point of $f$. Then there exists $r>0$
	such that $z_{0}$ is the only fixed point of $f$ in $\B(z_{0},r)$. Since
	$f_{k}^{m_k}\to f$ is uniformly on $\overline{\B}(z_{0},r)$, for sufficiently large
	$k$, we have
	\[
		\|f_{k}^{m_k}(z)-f(z)\|<\|f(z)-z\| \quad \text{for all }z\in\partial\B(z_{0},
		r).
	\]
	Hence, the map $f_{k}^{m_k}$ has no fixed points on $\partial\B(z_{0},r)$. Again by Result~\ref{R:vigue}, none of the eigenvalues of 
    $df(z_0)$ can be $1$. Therefore the multiplicity of $f(z) - z$ at $z_0$ is $1$.  
	By Rouché's theorem, Result~\ref{R:Rouché}, $f-I$ and
	$f_{k}^{m_k}-I$ have the same number of zeros, counted with multiplicities, in $\B(z_{0},r)$. Thus $f$ and $f
	_{k}^{m_k}$ have the same number of fixed points in $\B(z_{0},r)$. Since
	$f_{k}^{m_k}$ has at least two fixed points there, so does $f$, contradicting
	our assumption that $z_{0}$ is an isolated fixed point of $f$ of multiplicity $1$. By Result~\ref{R:vigue},
	it follows that the fixed point set $\F(f)$ contains a positive-dimensional
	complex submanifold passing through $z_{0}$ and that $1$ is an eigenvalue of
	  $df(z_{0})$. But we just doctored the eigenvalues $\lambda_{jk}
	^{m_k}$ to lie outside an arc of $\Sone$ containing $1$ so no eigenvalue
	of $df(z_{0})$ is $1$.
    This contradiction shows that for each point $a \in U$,
	there exists a small ball centered at $a$ such that each $f \in H$ has at most
	one fixed point in the small ball. We can cover $\overline{D}$ by finitely
	many of these balls and the number of balls needed gives us the necessary
	bound. This completes the proof of the second part of the theorem. 
	\qed
	\medskip

	\noindent\textbf{Proof of Theorem \ref{T:meta}}

  Let $f \in \A(D)$ be an automorphism that extends
	diffeomorphically to $\overline{D}$. We may, in fact, assume that $f$ extends as
	a non-singular smooth map to some domain $\OM \Supset \overline{D}$.
	Furthermore, $d_\R f(p)$ (the real-derivative map from $\R^{2n}$ to $\R^{2n}$) is a complex linear map whenever $p \in \bdy D$ as $f$ extends
	diffeomorphically to $\overline{D}$ and the Cauchy--Riemann equations hold on
	$D$.
	
	We will first show that if $p \in \bdy D$ is a fixed point of $f$ then
	$df(p)$ (which is just $d_\R f(p)$ treated as a complex linear map under the standard identification) has a positive real eigenvalue (see \cite{Liu2016} for stronger
	conclusions when $D$ is strongly pseudoconvex). Let
	$\rho$ be a smooth defining function for $D$. 
	It is easy to see that $\rho \circ f$ is also a defining function for $D$. By standard
	properties of defining functions, there exists a positive smooth
	function $h$ defined in a neighborhood of $p$ such that
	\[
		\rho \circ f = h\, \rho.
	\]
	Differentiating both sides at $p$, we obtain
	\[
		d_\R\rho(p) \circ d_\R f(p) = h(p)\, d_\R\rho(p).
	\]
    Passing to
	the dual map, we obtain
	\[
		d_\R f(p)^{*}(d_\R \rho(p)) = h(p)\, d_\R \rho(p).
	\]
	Thus $d_\R \rho(p)$ is an eigenvector of $d_\R f(p)^{*}$ with eigenvalue $h(p) > 0$. It
	follows that $d_\R f(p)$ has a positive real eigenvalue. Therefore, the positive real
	eigenvalue of $d_\R f(p)$ is also an eigenvalue of $df(p)$ as a complex linear map.

	We now show that there exists a compact subset $K \subset D$ such that, for
	each $f \in H$ with $\F(f)$ discrete, we have $\F(f) \subset K$. Suppose, to
	the contrary, that no such compact set exists. Then there are sequences $\{f_{k}
	\}$ of automorphisms in $H$ with discrete fixed point set and $z_{k} \in \F(f_{k})$ such that $z_{k} \to p \in \partial
	D$. Since each $f_{k}$ is an automorphism, all eigenvalues of $df_{k}(z_{k})$ lie
	on $\Sone$. As the fixed point set of $f_{k}$ is
	discrete, none of eigenvalues of $df_{k}(z_{k})$ are equal to $1$. By the
	arc-avoidance lemma (Lemma \ref{T:arc}), there exists an open arc $L \subset \Sone$ containing $1$,
	and a sequence of integers $m_{k} \in \mathbb{N}$ such that
	\[
		\lambda_{jk}^{\,m_k}\notin L \quad \text{for all }j = 1, \ldots, n \text{ and
		}k \in \mathbb{N},
	\]
	where $\{\lambda_{jk}\}_{j=1}^{n}$ are the eigenvalues of $df_{k}(z_{k})$. By
	the compactness of $H$ in the $C^{\infty}$-topology on $\overline{D}$, after passing to
	a subsequence, we may assume that
	\[
		f_{k}^{m_k}\longrightarrow f \quad \text{in the $C^\infty(\overline{D}$)-topology},
	\]
	for some $f \in \A(D)$ that extends smoothly to $\overline{D}$. Since $f_{k}^{m_k}(z_{k}) = z_{k}$
	and $z_{k} \to p$, it follows that $f(p) = p$. By the previous paragraph, the
	 $df(p)$ has a positive real eigenvalue. On the other hand,
	$df_{k}^{m_k}(z_{k}) \to df(p)$, and once again we have doctored the eigenvalues
	$\lambda_{jk}^{\,m_k}$ to lie outside the arc $L$ containing $1$. This is
	a contradiction. Therefore, there exists a compact set $K \subset D$ such that
	$\F(f) \subset K$ for every $f \in H$ with discrete fixed point set.

	It now follows from exactly the same argument as in the proof of Theorem~\ref{T:main}
	that for each $z \in K$, there exists a neighborhood $U$ of $z$ such that
	any $f \in H$ with discrete fixed point set has at most one fixed point in $U$. We conclude that there exists a
	uniform bound $m = m(D)$ on the number of fixed points of $f$. 
	\qed

	\section{Examples and comments}
	\label{S:examples}

	\subsection{Examples of domains that satisfy the hypotheses of the main theorems}

	We will now discuss several classes of domains that satisfy the hypotheses of our
	main results. For our main theorems, we require automorphisms to extend either
	to the boundary or past the boundary. We first discuss extension beyond the
	boundary. Result~\ref{R:extension} already gives a general class of domains in
	which we have extension of automorphisms beyond the boundary. The most general
	result of this type are due to Diederich, Fornæss and Pinchuk.

	\begin{result}
		[Diederich and Pinchuk \cite{Pinchuk1995,Pinchuk2003}, Diederich and Fornæss \cite{fornaess1988}] Let $D \subset \C^{n}$
		be a bounded domain with smooth, real-analytic boundary and let $f \in \A(D)$.
		Assume that either $n = 2$ or that $f$ extends continuously to
		$\overline{D}$ or that $D$ is pseudoconvex. Then $f$ extends holomorphically beyond the boundary.
	\end{result}

    Theorem~\ref{T:main} combined with the above result shows the first part of 
    Corollary~\ref{C:main2}. The above result combined with Theorem~\ref{thm: uniform extension} shows that if
    $H \subset \A(D)$ is compact and if $n = 2$ or $D$ is psedouconvex then any $f \in H$ extends to be an automorphism of a common 
    domain $U \Supset \overline{D}$. We note that this last conclusion for $n = 2$ is the main result of
    \cite{Kaushal2004} and the pseudoconvex case is the main result of \cite{Coupet96}. Bedford and Pinchuk \cite{Bedford98} gave a complete classification of bounded
    domains with smooth, real-analytic boundary in $\C^2$ with non-compact automorphism group. 
    They showed that such a domain is biholomorphic to a domain of the form 
    \[
        \{(z,w):|z|^{2k} + |w|^2 < 1\}.
    \]
    Since such a domain is convex, the second part of Corollary~\ref{C:main2} follows from Theorem~\ref{T:main}
    and Result~\ref{R:viguecvx}.

	In the non-smooth setting, the most general results are for domains that
	admits certain type of symmetries.

	\begin{definition}
		Let $G$ be a compact Lie group and let $G$ act linearly on $\mathbb{C}^{n}$,
		i.e., there exists a continuous representation
		$\rho : G \to \mathrm{GL}(\mathbb{C}^{n})$. Denote by $\hol(\mathbb{C}^{n})^{G}$,
		the set of $G$-invariant entire functions, i.e.,
		\[
			\hol(\mathbb{C}^{n})^{G} := \{ f \in \hol(\mathbb{C}^{n}) : f \circ \rho(g)
			= f \text{ for all }g \in G \}.
		\]
		A domain $D \subset \mathbb{C}^{n}$ is called \textit{$G$-invariant} if
		\[
			\rho(g)\cdot D = D \quad \text{for all }g \in G.
		\]
	\end{definition}

	We have the following result of Ning, Zhang, and Zhou that builds on the work of
	Bell \cite{Bell1982} and subsumes a number of previous extension results.

	\begin{result}[Ning, Zhang, and Zhou \cite{Ning2017}]\label{R:Ning}
		 Suppose $G_{1}$ and $G_{2}$ are compact
		Lie groups acting linearly on $\mathbb{C}^{n}$ with $O(\mathbb{C}^{n})^{G_j}=
		\mathbb{C}$, and $0 \in \Omega_{j} \subset \mathbb{C}^{n}$ are $G_{j}$-invariant
		domains $(j=1,2)$. If $f : \Omega_{1} \to \Omega_{2}$ is a proper holomorphic
		mapping, then $f$ extends holomorphically to a neighborhood of
		$\overline{\Omega_1}$.
	\end{result}

	\begin{remark}
		Circular and Reinhardt domains that contain the origin are $G$-invariant. So
		are quasi-balanced domains. i.e., domains $D$ such that there exist positive
		integers $m_{1},\ldots,m_{n}$ such that for every $z=(z_{1},\ldots,z_{n})\in
		D$ and every $\lambda \in \overline{\mathbb{D}}$, we have
		\[
			(\lambda^{m_1}z_{1},\ldots,\lambda^{m_n}z_{n}) \in D .
		\]
	\end{remark}
    It is clear from the above remark that Corollary~\ref{C:main} follows from Theorem~\ref{T:main}.

	We now discuss results pertaining to the hypotheses in Theorem~\ref{T:meta}.
	The most general framework for studying smooth extension to the boundary is
	using a certain regularity condition on the Bergman projection called \textit{condition
	$R$} formulated by Bell and Ligocka \cite{Bell1980}.

	\begin{definition}[Condition $R$]
		Let $D \subset \C^{n}$ be a bounded domain with smooth boundary. We say that
		$D$ satisfies condition $R$ if the Bergman projection $P$ maps $C^{\infty}(\overline
		{D})$ into $C^{\infty}(\overline{D})$.
	\end{definition}

	We have the following major result of Bell and Ligocka which was inspired by
	the landmark work of Fefferman \cite{Fefferman1974}.

	\begin{result}[Bell and Ligocka \cite{Bell1980}]\label{R:Ligocka}
		 Let $D \subset \C^{n}$ be a bounded domain
		with smooth boundary that satisfies condition $R$. Then any $f \in \A(D)$
		extends to be a diffeomorphism of $\overline{D}$.
	\end{result}

	The following is a list of some general classes of bounded domains that satisfy
	condition $R$:
	\begin{enumerate}
		\item Strongly pseudoconvex domains \cite{Kohn1964}.

		\item Pseudoconvex domains with real-analytic boundary \cite{Fornaess1978, Kohn1979}.

		\item Smoothly bounded domains that admit a plurisubharmonic defining function
			\cite{Boas1991}.

		\item Smoothly bounded pseudoconvex domains with finite type boundary in the sense of D'Angelo
			\cite{Catlin1987}.
	\end{enumerate}

	The above list shows that there are plenty of domains that satisfy the extensibility
	hypothesis in Theorem~\ref{T:meta}. The following result due to Barrett
	\cite{Barrett1981} (see also \cite{Bell1987}) gives a
	sufficient condition for the extended automorphisms to have good convergence
	behavior.

	\begin{result}
    Let $D \subset \C^{n}$ be a bounded domain with smooth boundary 
		that satisfies condition $R$. Let $f_{i} \in \A(D)$ be a sequence such that $f_{i} \to f \in \A(D)$. Then $f
		_{i} \to f$ in the $\smoo^{\infty}(\overline{D})$-topology.
	\end{result}

	Thus, the classes of domains in the above list are natural classes of domains for
	which Theorem~\ref{T:meta} holds. Corollary~\ref{C:meta} says more in the
	setting of strongly pseudoconvex domains.

	\medskip

	\noindent
	\textbf{Proof of Corollary~\ref{C:meta}}

	By the Wong--Rosay theorem \cite{Wong1977,Rosay1979}, either $\A(D)$ is
	compact or $D$ is biholomorphic to the unit ball. If $D$ is biholomorphic to the
	unit ball, then Result~\ref{R:viguecvx} immediately shows $m$ can be taken to
	be $1$. So we can assume that $\A(D)$ is compact and Theorem~\ref{T:meta} gives
	us the required bound $m$. 
	\qed

  \noindent \textbf{Proof of Corollary~\ref{C:meta2}}

  Let $z_0 \in \F(f)$. The corollary now follows from taking $H$ to be the
  isotropy subgroup of $z_0$ in Theorem~\ref{T:meta}. That the isotropy group is compact follows from a simple argument involving Montel's theorem and Result~\ref{R:cartan}.

	\subsection{Comments on generalizations to more general maps}

	We now briefly discuss some possible generalizations of our main theorems to
	more general classes of holomorphic maps. 
	\begin{enumerate}
		\item The proof of the first part of Theorem~\ref{T:main} works
		\textit{mutatis mutandis} for a general holomorphic self-map that
		extends holomorphically past the boundary. 
		\item Result~\ref{R:Ning} is actually stated for proper holomorphic self-maps of
		$G$-invariant domains. So if $D$ is a $G$-invariant domain and $f$ is a
		proper holomorphic self-map whose fixed point  is
		discrete then the fixed point set is finite.
		\item Our proof of Theorem~\ref{T:meta} used a very simple version of
		Schwarz lemma at the boundary. If $D$ is smoothly bounded and strongly
		pseudoconvex then a more delicate analysis (see \cite{Liu2016,Wang2017}) reveals
		that if $f \in \hol(D,D)$ and extends smoothly to the boundary near a point 
        $p \in \bdy D$ with $f(p) = p$, then one of the eigenvalues of $df(p)$ must be positive. 
        Thus the arguments in Theorem~\ref{T:meta} could be of possible help in showing $\F(f)$ is
		finite. One pertinent observation is that if such an $f$ has a discrete fixed point set that is
		\textbf{not} a singleton then none of the fixed points can be
		\textit{attracting}. Of course, the key obstacle is to formulate a natural
		compactness condition. One possible (unnatural) condition is to assume that the closed cyclic group generated by $f$ in $\A(D)$ is compact in the $\smoo^\infty(\overline{D})$-topology.
		\item Result~\ref{R:Ligocka} also holds for proper holomorphic self-maps
		of smoothly bounded pseudoconvex domains that satisfy condition $R$
		\cite{BC82,DF82}.  It is natural to investigate when one has a
		uniform bound on the number of fixed points of proper holomorphic
		self-maps.  One needs to
		study the compactness of the family of proper holomorphic self-maps and
		uniform boundary regularity. The following papers might be of interest;
		\cite{Bell1988, Coupet96, Klingenberg1990,Klingenberg1991}.  
	\end{enumerate}

    \section*{Acknowledgements}
    The authors thank Dr. G. P. Balakumar for his insightful comments that helped to clarify and improve this article. The first author also thanks his friends, Dr. Jiju Mammen and Dr. Agniva Chatterjee, for their helpful discussions on key concepts in Several Complex Variables.

	\bibliographystyle{amsalpha}

	\bibliography{isolated}
\end{document}